\newcommand{\Rr}{\enm{\cal{R}}}
\newcommand{\Ss}{\enm{\cal{S}}}

\newcommand{\Uu}{\enm{\cal{U}}}

\newcommand{\Xx}{\enm{\cal{X}}}

\renewcommand{\phi}{\varphi}
\renewcommand{\theta}{\vartheta}
\documentclass[12pt,oneside,reqno]{amsart}
 
\usepackage[all]{xy}                        %
\usepackage[hyperindex=true, colorlinks=true, linkcolor=blue, filecolor=black,
citecolor=blue, urlcolor=blue, bookmarks=true, pagebackref=true,
pdftitle={COUNTING ACM TORIC BUNDLES OF RANK 2
ON SMOOTH TORIC SURFACES}]{hyperref}

\usepackage[margin=1in]{geometry}
\usepackage[table]{xcolor}
\usepackage{colortbl, array, booktabs}

\definecolor{lightred}{rgb}{1, 0.85, 0.85}
\definecolor{lightblue}{rgb}{0.85, 0.9, 1}
\definecolor{darkline}{rgb}{0.4, 0.4, 0.4} % 선 색 더 진하게
\usepackage{verbatim}
\usepackage[numbers]{natbib}
\CompileMatrices                            % Faster

\UseTips                                    % Use

\usepackage[bookmarks=true]{hyperref}       % Hyperref
\hypersetup{
    colorlinks=true,    % 링크에 색상 적용
    linkcolor=blue,     % 문서 내 링크 색상
    citecolor=blue,     % 참고문헌 인용 링크 색상
    filecolor=blue,     % 파일 링크 색상
    urlcolor=blue       % URL 링크 색상
}
\usepackage{amssymb,latexsym,amsmath,amscd}
\numberwithin{equation}{section}
\usepackage{xspace}
\usepackage{color}
\usepackage{graphicx}
\usepackage{caption}
\usepackage{mathtools}
\usepackage{tikz}
\usetikzlibrary{cd}
\usepackage[utf8]{inputenc}
\usepackage{fourier}
\usepackage{array}
\usepackage{makecell}
\usepackage{colortbl}
\usepackage{bm}
\usepackage{array}
\usepackage{subcaption}
\usepackage[table]{xcolor}
\usepackage{multirow}

\definecolor{lightred}{RGB}{255,220,220}
\definecolor{lightblue}{RGB}{220,235,255}

\newcolumntype{L}[1]{>{\raggedright\let\newline\\\arraybackslash\hspace{0pt}}m{#1}}
\newcolumntype{C}[1]{>{\centering\let\newline\\\arraybackslash\hspace{0pt}}m{#1}}
\newcolumntype{R}[1]{>{\raggedleft\let\newline\\\arraybackslash\hspace{0pt}}m{#1}}
\definecolor{lightred}{rgb}{1.0,0.9,0.9}
\definecolor{lightblue}{rgb}{0.9,0.9,1.0}

\reversemarginpar

\theoremstyle{plain}
\newtheorem{theorem}{Theorem}[section]
\newtheorem*{theorem*}{Theorem}
\newtheorem{proposition}[theorem]{Proposition}
\newtheorem{corollary}[theorem]{Corollary}
\newtheorem{lemma}[theorem]{Lemma}
\newtheorem{conjecture}[theorem]{Conjecture}

\newtheorem{claim}{Claim}[theorem]
\newenvironment{claimproof}[1][Proof of claim]{\begin{proof}[#1]}{\end{proof}}

\theoremstyle{definition}
\newtheorem{definition}[theorem]{Definition}

\newtheorem{remark}[theorem]{Remark}

\newcommand{\enm}[1]{\ensuremath{#1}}          %
\newcommand{\op}[1]{\operatorname{#1}}
\newcommand{\cal}[1]{\mathcal{#1}}

\newcommand{\CC}{\enm{\mathbb{C}}}

\newcommand{\ZZ}{\enm{\mathbb{Z}}}

\newcommand{\PP}{\enm{\mathbb{P}}}

\newcommand{\Aa}{\enm{\cal{A}}}
\newcommand{\Bb}{\enm{\cal{B}}}

\newcommand{\Ee}{\enm{\cal{E}}}
\newcommand{\Ff}{\enm{\cal{F}}}

\newcommand{\Ii}{\enm{\cal{I}}}

\newcommand{\Ll}{\enm{\cal{L}}}

\newcommand{\Oo}{\enm{\cal{O}}}
\renewcommand{\epsilon}{\varepsilon}

\newcommand{\Hom}{\op{Hom}}

\newcommand{\Ext}{\op{Ext}}

\newcommand{\End}{\op{End}}

\newcommand{\id}{\op{id}}

\newcommand{\Filt}{\mathbf{Filt}}

\renewcommand{\to}[1][]{\xrightarrow{\ #1\ }}

\newcommand{\old}[1]{}

\makeatletter
\NewCommandCopy\@@pmod\pmod
\DeclareRobustCommand{\pmod}{\@ifstar\@pmods\@@pmod}
\def\@pmods#1{\mkern4mu({\operator@font mod}\mkern 6mu#1)}
\makeatother

\newcommand{\Rep}{\mathbf{Rep}}
\newcommand{\Ulr}{\mathbf{Ulr}}

\newcommand{\aCM}{\mathbf{aCM}}

\begin{document}
\renewcommand{\arraystretch}{1.4}
\arrayrulecolor{darkline}

\title[Toric representation type of the Veronese surface]{Toric representation type of the Veronese surface}

\author{Yeonjae Hong and Sukmoon Huh}

\address{Gyeongsang National University, Jinju 52828, Korea}
\email{hnjae529@gmail.com}

\address{Sungkyunkwan University, Suwon 440-746, Korea}
\email{sukmoonh@skku.edu}

\thanks{YH was supported by the Global--Learning \& Academic research institution for Master's$\cdot$PhD students, and Postdocs (LAMP) Program of the National Research Foundation of Korea (NRF), funded by the Ministry of Education (No.~RS-2023-00301974). SH was supported by the National Research Foundation of Korea (NRF) grant funded by the Korean government (MSIT) (No.~RS-2023-00208874).}

%\author{Yeon-Jae Hong} 
%\address{Sungkyunkwan University, Suwon 440-746, Korea}
%\email{mathbunny529@skku.edu}
%\author{Author2}
%\address{Sungkyunkwan University, Suwon 440-746, Korea}
%\email{email2}


\begin{abstract}
In this article we determine the toric representation type of the Veronese surface $\bigl( \PP^2, \Oo_{\PP^2}(d) \bigr)$. Based on Klyachko filtrations, we introduce an explicit criterion for a toric vector bundle of arbitrary rank to be arithmetically Cohen–Macaulay. For $d\ge 3$, suitable configurations of partial flags produce stable toric $d$-aCM bundles corresponding to imaginary non-isotropic Schur roots of star-shaped quivers. Their self-extensions give an exact representation embedding of $\mathrm{mod}\CC \langle x,y\rangle$, proving that the Veronese surface is toric-wild precisely for $d\ge 3$, while it is toric-finite for $d=1,2$. For $d=3,4$, suitable twists of the basic stable bundles are Ulrich, and the same construction proves that the corresponding Veronese surfaces are toric Ulrich-wild.
\end{abstract}

\maketitle

\section{Introduction}
Arithmetically Cohen–Macaulay vector bundles occupy a central position at the intersection of algebraic geometry, commutative algebra, and representation theory. Let $(X,H)$ be a polarized projective variety of dimension $n$. A vector bundle $\Ee$ on $X$ is called arithmetically Cohen–Macaulay (for short, $H$-aCM) if
\[
\mathrm{H}^i\bigl(X,\Ee \otimes \Oo_X(tH)\bigr)=0 \qquad \text{for every $t\in\ZZ$ and $0<i<n$.}
\]
The classification of indecomposable aCM bundles is closely related to the classification of maximal Cohen–Macaulay modules over the homogeneous coordinate ring of X. As in the representation theory of finite-dimensional algebras, one is naturally led to distinguish varieties of finite, tame, and wild Cohen–Macaulay representation type. Roughly speaking, finite type means that there are only finitely many indecomposable aCM bundles up to twist, tame type means that indecomposables occur essentially in families of dimension at most one, and wild type indicates that their classification contains problems of arbitrarily large complexity. This point of view has led to substantial progress in the study of aCM bundles on hypersurfaces, varieties of minimal degree, del Pezzo surfaces, Segre varieties, and several homogeneous varieties; see, for example, \cite{BGS87, CH11, CMR16, CMP12, EH88, FM, FM17, FPL, Kno87}.

When $X$ is a toric variety, it is natural to refine this problem by imposing equivariance. A toric vector bundle is a vector bundle endowed with a compatible action of the dense torus. Although equivariance is a restrictive condition, the category of toric bundles can still exhibit unexpectedly complicated behavior. At the same time, it admits a particularly effective combinatorial description through Klyachko’s theory: a toric vector bundle is encoded by a finite collection of $\ZZ$-filtrations of a fixed vector space satisfying a compatibility condition determined by the fan. Thus the cohomology and decomposability of toric bundles can be studied through configurations of subspaces and flags. This makes the toric setting especially suitable for investigating representation type by explicit geometric and combinatorial methods.

The purpose of this paper is to determine the toric representation type of the Veronese surface $\bigl(\PP^2,\Oo_{\PP^2}(d)\bigr)$. Equivalently, we study toric vector bundles $\Ee$ on $\PP^2$ satisfying $\mathrm{H}^1\bigl(\Ee(dt)\bigr)=0$ for every $t\in \ZZ$. We refer to such bundles as toric $d$-aCM bundles. We identify bundles up to twists by powers of $\Oo_{\PP^2}(d)$ and up to changes of equivariant structure by characters of the torus. The latter identification is necessary because a fixed underlying vector bundle may admit infinitely many equivariant structures obtained by character twists.

Our principal result is the following sharp dichotomy.
\begin{theorem}\label{toric-wild}
The polarized toric surface $\bigl(\PP^2,\Oo_{\PP^2}(d)\bigr)$ is toric wild if and only if $d\ge 3$. In cases of $d=1,2$, it is of toric-finite representation type. 
\end{theorem}
\noindent Here, toric wildness is understood in the representation-embedding sense: there exists a $CC$-linear exact functor
\[
\mathrm{mod}\CC\langle x,y\rangle \longrightarrow \Rep(Q).
\]
that preserves indecomposable objects and reflects isomorphism classes, even after passing to the equivalence relation generated by polarization and character twists. Since finite-dimensional $\CC\langle x,y, \rangle$-modules are arbitrary pairs of matrices considered up to simultaneous similarity, this establishes the full representation-theoretic wildness of the toric $d$-aCM category.

A key ingredient in the proof is a cohomological criterion expressed in terms of Klyachko filtrations. For three subspaces $A_0,A_1,A_2$ of a vector space, we introduce the distributivity defect
\[
\Delta(A_0,A_1,A_2)=\dim \frac{A_0\cap(A_1+A_2)} {(A_0\cap A_1)+(A_0\cap A_2)}.
\]
We observe that the weight components of $\mathrm{H}^1(\Ee)$ are precisely measured by these defects. Consequently, a toric bundle $\Ee$ is $d$-aCM if and only if the defect vanishes for every triple of filtration stages whose indices have sum divisible by $d$; see Corollary \ref{dacm}. This criterion extends the analysis in \cite{HH} for rank two to bundles of arbitrary rank. An important feature in higher rank is that the dimensions of the filtration stages alone no longer determine the cohomology: their relative incidence positions must also be taken into account. On the general-position locus, however, the defect admits an explicit dimension formula that is homogeneous under simultaneous scaling.

Using this criterion, we construct suitable configurations of three partial flags. For $d=3$, the relevant bundles have rank $6n$ and are parametrized by
\[
\mathrm{Fl}(n,3n,5n;6n)^2\times\mathrm{Fl}(2n,4n;6n),
\]
whereas, for $d\geq4$, we consider $\mathrm{Fl}(n,2n,3n;4n)^3$. The filtration indices are chosen so that every triple with positive distributivity defect has total index nonzero modulo $d$. The resulting bundles are therefore $d$-aCM; see Lemmas \ref{d3-filtration-acm} and \ref{dge4-filtration-acm}. 

These flag configurations can be interpreted as representations of suitable star-shaped quivers $Q_d$. The relevant dimension vector $\alpha_d$ is an imaginary non-isotropic Schur root satisfying the Tits form $q_{Q_d}(\alpha_d)=-2$. By Schofield’s homogeneity theorem for canonical decompositions, every multiple $n\alpha_d$ is again a Schur root. Consequently, there is a nonempty open locus of flag configurations whose common endomorphism algebra consists only of scalar maps. The corresponding toric bundles are equivariantly simple and hence indecomposable. After quotienting by simultaneous change of basis, we obtain families of pairwise inequivalent indecomposable toric $d$-aCM bundles of dimension $2n^2+1$. This proves that the Veronese surface is geometrically toric-wild for every $d\ge 3$; see Corollary \ref{geom-toric-wild}.

We also study the stability of the basic Schur bundles $\Ee_d$ associated with the dimension vectors $\alpha_d$. By applying Klyachko’s slope criterion and analyzing the possible Schubert positions of a subspace with respect to three general flags, we prove that $\Ee_d$ is slope-stable with respect to $\Oo_{\PP^2}(1)$ for every $d\ge 3$. In the case $d=3$, the required inequality is obtained from the incidence conditions for two flags of type $(1,3,5)$ and one flag of type $(2,4)$ in a six-dimensional vector space. For $d\ge 4$, it follows from the corresponding calculation for three complete flags in dimension four. Since stability is preserved by tensoring with a line bundle, the normalized bundles used below are stable as well; see Proposition \ref{d3-stable}. 

We subsequently strengthen this geometric statement to toric wildness in the representation--embedding sense. Fixing a Schur representation $R_d$ of dimension vector $\alpha_d$, we consider the extension-closed subcategory $\Filt_{Q_d}(R_d)$ consisting of quiver representations admitting a finite filtration whose successive quotients are isomorphic to $R_d$. Since the arrow maps of $R_d$ are injective, every object of this category is again a flag representation and hence determines a toric vector bundle. Moreover, the associated bundle is an iterated equivariant self-extension of the $d$-aCM bundle corresponding to $R_d$, and is therefore itself $d$-aCM.

The equality $q_{Q_d}(\alpha_d)=-2$, together with $\End_{Q_d}(R_d)=\CC$, gives $\dim_\CC \Ext_{Q_d}^1(R_d, R_d)=3$. Using three linearly independent self-extension classes, we construct an exact functor
\[
\mathrm{mod}\CC\langle x,y\rangle \longrightarrow \Filt_{Q_d}(R_d)
\]
in Theorem \ref{embedding}. The two endomorphisms defining a $\CC\langle x,y\rangle$-module are encoded in the off-diagonal terms of a self-extension of copies of $R_d$. We prove that this functor is faithful, preserves indecomposability, and reflects isomorphism classes. Finally, a comparison of ranks, first Chern classes, and normalized equivariant determinants shows that it continues to reflect equivalence after allowing twists by $\Oo_{\PP^2}(d)$ and torus characters; see Theorem \ref{main2}.

The stability result also allows us to strengthen the conclusion for the Veronese surfaces for $d=3,4$. With a suitable twist of $\Ee_d$ as an Ulrich bundle, we obtain a stronger wildness statement.

\begin{theorem}\label{ulrich-wild}
The Veronese surface $\bigl(\PP^2,\Oo_{\PP^2}(d)\bigr)$ for $d=3,4$ is toric Ulrich-wild.
\end{theorem}

The paper is organized as follows. Section 2 recalls equivariant sheaves on toric varieties, Klyachko filtrations, the relevant notions of toric representation type, and basic results concerning Schur roots and canonical decompositions of quiver representations. In Section 3, we derive the cohomological criterion in terms of the distributivity defect and prove toric finiteness for $d=1,2$. Section 4 constructs unbounded families of indecomposable toric $d$-aCM bundles, proves geometric toric wildness for $d\ge 3$, and establishes the slope stability of the basic Schur bundles. Finally, Section 5 constructs the representation embedding, proves toric wildness for every $d\ge 3$, establishes toric Ulrich-wildness for d=3,4, and formulates the conjecture for $d>4$.

%%%%%%%%%%%%%%%%%%%%%%%%%%%%%%%%%
\section{Preliminaries}
An $n$-dimensional  \textit{toric variety} is an irreducible variety $X$ containing a torus $T \simeq \left(\mathbb{C}^{\times}\right)^{n}$ as a Zariski open subset such that the action of $T$ on itself extends to an algebraic action of $T$ on $X$. Given $m = (a_{1}, \dots , a_{n}) \in \mathbb{Z}^n$, the character $\chi^{m}: (\mathbb{C}^{\times})^{n}\rightarrow \mathbb{C}^{\times}$ is defined by 
\[
\chi^{m}(t_{1},\dots,t_{n}) = {t_{1}}^{a_1} \dots {t_{n}}^{a_{n}}.
\]
The set of characters forms a free abelian group $M$ of rank $n$. A \textit{one-parameter subgroup} of $T$ is a morphism $\ u: \mathbb{C}^{\times} \rightarrow T$ that is a group homomorphism. The one-parameter subgroups form a free abelian group $N$ of rank $n$. We define $N_{\mathbb{R}} := N \otimes_{\mathbb{Z}} \mathbb{R}$ and $M_{\mathbb{R}} := M \otimes_{\mathbb{Z}} \mathbb{R}$. There is a natural pairing \( \langle \cdot, \cdot \rangle \colon M \times N \to \mathbb{Z} \), defined by the relation
\[
m \circ u(t) = t^{\langle m, u \rangle},
\]
for all \( t \in \mathbb{C}^\times \), where \( m \in M \) and \( u \in N \). Given a fan $\Sigma$ in $N_{\mathbb{R}}$, each cone $\sigma\in \Sigma$ produces an affine toric variety $U_{\sigma}= \operatorname{Spec}(\mathbb{C}[S_{\sigma}])$ for the semigroup $S_{\sigma}=\sigma^\vee \cap M$, gluing together to form a toric variety $X$ with $U_\sigma$ as a \( T \)-invariant affine chart of $X$; see \cite{Ful93} and \cite[Chapter 3]{CLS11}. There exists a one-to-one correspondence between the set of rays $\rho \in \Sigma(1)$ and the set of $T$-invariant Weil divisors $D_\rho$. For the remainder of this section, we assume that $X=X_\Sigma$ is smooth and complete. 

Given a cone \( \sigma \), one can define a preorder $\preceq_{\sigma}$ on \( M \) via
\[
m \preceq_{\sigma} m' \quad \text{if and only if} \quad m' - m \in S_{\sigma}.
\]
This allows us to define filtrations and associated vector spaces.

\begin{definition}
A \textit{\(\sigma\)-family} $\hat{E}^{\sigma}$ consists of a collection of vector spaces $\left\{E^\sigma_{ m} \right\}_{m \in M}$ equipped with transition maps
\[
\chi^\sigma_{m, m'}: E^\sigma_{m} \to E^\sigma_{m'}
\]
for each \( m \preceq_{\sigma} m' \), satisfying the conditions:
\begin{enumerate}
    \item \( \chi^\sigma_{m, m} \) is the identity map.
    \item If \( m \preceq_{\sigma} m' \preceq_{\sigma} m'' \), then \( \chi^\sigma_{m, m''} = \chi^\sigma_{m', m''} \circ \chi^\sigma_{m, m'} \).
\end{enumerate}
\end{definition}

\noindent Every $T$-equivariant quasi-coherent sheaf over $ U_{\sigma}$ gives rise to a corresponding $\sigma$-family through its isotypical decomposition. Specifically, one sets
\[
E^\sigma_{ m} = \Gamma(U_{\sigma}, \mathcal{E})_m,
\]
where the right-hand side denotes the \( m \)-th isotypical component. The transition maps are given by multiplication with monomials in the coordinate ring $\mathbb{C}[S_{\sigma}]$. According to \cite[Theorem 4.5]{Per03}, the category of $T$-equivariant quasi-coherent sheaves over $U_{\sigma}$ is equivalent to the category of $\sigma$-families. To describe a $T$-equivariant quasi-coherent sheaf in terms of filtrations on an abstract toric variety, we consider the following definition, which corresponds to the gluing of $\sigma$-families. 
\begin{definition}\cite[Definition 4.8]{Per03} A collection $\left\{ \hat{E}^{\sigma} \right\}_{\sigma \in \Sigma}$ of $\sigma$-families is called a $\Sigma$-\textit{family}, if for each pair $\tau \prec \sigma$ with inclusions $i_{\sigma}^{\tau} : U_{\tau} \hookrightarrow U_{\sigma}$ there exists an isomorphism of families $\eta_{\tau\sigma} : i_{\sigma}^{\tau*} \hat{E}^{\sigma} \cong \hat{E}^{\tau}$ such that for each triple $\rho \prec \tau \prec \sigma$, the following equality holds: $\eta_{\rho\sigma} = \eta_{\rho\tau} \circ {i_{\tau }^{\rho}}^* \eta_{\tau\sigma}$.
\end{definition}

If the sheaf \(\mathcal{E}\) is torsion-free, all vector spaces within the \(\Sigma\)-family can be regarded as subspaces of the fiber of \(\mathcal{E}\) at the identity point of the torus. Throughout the paper, we denote the fiber of a sheaf \(\mathcal{E}\) at the identity point by \(E\), using roman font to distinguish it from the sheaf itself. This convention also applies to other sheaves such as \(\mathcal{F}\), \(\mathcal{G}\), etc.  
Further, if the sheaf is reflexive, the \(\sigma\)-families for \(\sigma \in \Sigma(1)\) determine the rest; see \cite[Theorem 4.21]{Per03}.

Denote the set of rays by $\{\rho_0, \dots, \rho_q\}$, and let $u_i$ denote the primitive vector of $\rho_i$. For a fixed vector space $E$, we consider a collection of decreasing $\mathbb{Z}$-filtrations $\{E^{\rho_i}(\bullet) \}_{{\rho_i} \in \Sigma(1)}$ with 
\[
E^{\rho_i}(\bullet): \quad E\supseteq \dots \supseteq E^{\rho_i}(-1) \supseteq E^{\rho_i}(0) \supseteq E^{\rho_i} (1) \dots \supseteq 0
\]
for each ray $\rho_i\in \Sigma(1)$. The filtrations
$\{E^{\rho_i}(\bullet)\}_{\rho_i\in\Sigma(1)}$
are said to be {\it full} if, for each $\rho_i\in\Sigma(1)$,
$E^{\rho_i}(j)=E$ for $j\ll0$ and $E^{\rho_i}(j)=0$ for $j\gg0$.
Unless otherwise specified, all filtrations are assumed to be full. Then Klyachko presented a theorem to express $T$-equivariant reflexive sheaves as a set of decreasing $\ZZ$-filtrations.

\begin{theorem}\textnormal{\cite[Theorem 1.3.2]{Kly91}}\label{kly} The category of $T$-equivariant reflexive sheaves on $X$ is equivalent to the category of decreasing $\mathbb{Z}$-filtrations $\{E^{\rho_i} (\bullet)\}_{{\rho_i} \in \Sigma(1)}$. Such a reflexive sheaf is a toric bundle if and only if the corresponding filtrations satisfy the following compatibility condition:

\leftskip=2em 
\noindent $\mathrm{(C)}$ For any cone $\sigma \in \Sigma$, the filtrations $\{E^{\rho_i} (\bullet)\}_{{\rho_i} \in \sigma(1)}$, consist of coordinate subspaces of some basis of $E$. 
\rightskip=3em
\end{theorem}

\noindent Let $X$ be a smooth toric variety with a polarization $H$. 
\begin{definition}
A $T$-equivariant vector bundle $\Ee$ on $X$ is called \emph{arithmetically Cohen--Macaulay} with respect to $H$, or simply \emph{$H$-aCM}, if
\[
 \mathrm{H}^i\bigl(X,\Ee \otimes \Oo_X(tH)\bigr)=0
 \qquad
 \text{for every $t\in\mathbb Z$ and $0<i<n$.}
\]
The exact category of $T$-equivariant $H$-aCM vector bundles is denoted by $\mathrm{aCM}_T(X,H)$.
\end{definition}

For a character $\chi\in M$, let $\CC_\chi$ denote the trivial line bundle endowed with the $T$-linearization determined by $\chi$.

\begin{definition}
For $\Ee,\Ff\in\mathrm{aCM}_T(X,H)$, write $\Ee\sim_{H}\Ff$ if there exist $q\in\mathbb \ZZ$ and $\chi\in M$ such that
\[
\Ff\cong_T \Ee\otimes \Oo_X(qH)\otimes\CC_\chi.
\]
Thus we identify both the usual degree shifts and changes of equivariant structure by a character of the torus. Note that the quotient by character twists is essential; without it, a single underlying bundle generally gives infinitely many equivariant bundles by changing its linearization.
\end{definition}

\begin{definition}
Let $S$ be a variety on which $T$ acts trivially. An \emph{algebraic family} of toric $H$-aCM bundles parametrized by $S$ is a $T$-equivariant vector bundle $\Ee$ on $X\times S$ such that
\[
 \Ee_s:=\Ee_{|{X\times\{s\}}}
\]
is $H$-aCM for every $s\in S$. The family consists of pairwise non-equivalent indecomposables if the underlying vector bundle of $\Ee_s$ is indecomposable for each $s\in S$ and
\[
s\ne s'\quad\Longrightarrow\quad \Ee_s\not\sim_{H}\Ee_{s'}.
\]
\end{definition}

\begin{definition}
The polarized toric variety $(X,H)$ is of \emph{toric-finite} representation type if the set
\[
 \left\{
 \begin{array}{c}
 \text{indecomposable objects of }\mathrm{aCM}_T(X,H)
 \end{array}
 \right\}\big/\!\sim_H
\]
is finite. The polarized variety is of \emph{toric-tame} representation type if it is not toric-finite and, for every positive integer $r$, the indecomposable objects of rank $r$ in $\aCM_T(X,H)$, modulo $\sim_H$, are covered by finitely many algebraic families $\Ee_j\rightarrow X\times S_j$ with $\dim S_j\leq 1$, and possibly finitely many exceptional isomorphism classes. Finally, $(X,H)$ is of \emph{toric-wild} representation type if there exists a $\CC$-linear exact functor
\[
\Phi:\operatorname{mod}\CC \langle x,y\rangle \longrightarrow \aCM_T(X,H)
\]
such that
\begin{enumerate}
\item $M$ indecomposable implies that $\Phi(M)$ is indecomposable;
\item for all finite-dimensional $\CC\langle x,y\rangle$-modules $M,N$,
\[
  \Phi(M)\sim_H \Phi(N)\quad\Longleftrightarrow\quad M\cong N.
\]
\end{enumerate}
In other words, the functor $\Phi$ is a representation embedding.
\end{definition}

\begin{remark}
The category $\mathrm{mod}\CC\langle x,y\rangle$ is the classical test object for wildness: finite-dimensional modules are arbitrary pairs of matrices up to simultaneous similarity. This viewpoint goes back to Nazarova's work on infinite-type quivers and to Drozd's tame--wild theory; see \cite{Nazarova,Drozd}. Requiring an exact functor that preserves indecomposables and reflects isomorphisms is the corresponding representation-embedding formulation. In the toric setting we additionally pass to $\sim_{H}$, since polarization twists and changes of linearization should not create new representation types; compare the Cohen--Macaulay formulation in \cite{DrozdGreuel}.
\end{remark}

\begin{comment}
\begin{remark}
One may refine the tame case as follows.
\begin{itemize}
\item $(X,H)$ is of \emph{toric-discrete type} if there are infinitely many classes but no positive-dimensional families of pairwise non-equivalent indecomposables.
\item $(X,H)$ is of \emph{properly toric-tame type} if one-dimensional families occur but all indecomposables are covered, rank by rank, by finitely many families of dimension at most one and finitely many exceptions.
\end{itemize}
\end{remark}
\end{comment}

\begin{remark}
The polarized toric variety $(X,H)$ is called \emph{geometrically toric-wild}, if for every integer $N>0$ there exists an algebraic family $\Ee\rightarrow X\times S$ with $\dim S\geq N$ of pairwise non-equivalent indecomposable toric $H$-aCM bundles. Then toric-wildness in the representation-embedding sense implies geometric toric wildness; see \cite[Definition 1.4]{DrozaGreuel2} and \cite[Remark 3.2]{KleppeMiroRoig}
\end{remark}

The following is the homogeneity theorem for canonical decomposition by Schofield; see \cite[Theorem 3.8]{Scho} and \cite[Remark 2 and Theorem 3]{DW}. 
\begin{theorem}[Schofield]\label{schofield-homogeneity}
Let $Q$ be an acyclic quiver and let
\[
\mathrm{can}(\alpha)=\beta_1\oplus\cdots\oplus\beta_s
\]
be the canonical decomposition of a dimension vector $\alpha$, where each $\beta_i$ is a Schur root and a general representation of dimension $\alpha$ is a direct sum of representations of dimensions $\beta_1,\ldots,\beta_s$. For a Schur root $\beta$, define
\[
\beta^{[n]}=
\begin{cases}
n\beta,&\beta\text{ is imaginary and non-isotropic},\\[2mm]
\underbrace{\beta\oplus\cdots\oplus\beta}_{n\text{ copies}},&\beta\text{ is real or isotropic}.
\end{cases}
\]
Then, for every integer $n\geq1$,
\[
\mathrm{can}(n\alpha)= \beta_1^{[n]}\oplus\cdots\oplus\beta_s^{[n]}.
\]
In particular, if $\beta$ is an imaginary non-isotropic Schur root, then $\mathrm{can}(n\beta)=n\beta$ is a single canonical summand. Consequently, $n\beta$ is again a Schur root for every $n\geq1$.
\end{theorem}

\noindent Here, for the Tits form defined below, a Schur root is called \emph{real} when $q(\alpha)=1$, \emph{isotropic imaginary} when $q(\alpha)=0$, and \emph{non-isotropic imaginary} when $q(\alpha)<0$. For the following computation, refer to \cite{DW-book}. 

\begin{lemma}\label{euler-hom-ext}
Let $Q=(Q_0,Q_1)$ be a finite acyclic quiver. For dimension vectors
$\alpha,\beta\in\ZZ^{Q_0}$, define the Euler form by
\[
\langle\alpha,\beta\rangle_Q=\sum_{v\in Q_0}\alpha(v)\beta(v)-\sum_{a:v\rightarrow w\in Q_1}\alpha(v)\beta(w).
\]
If $M,N$ are finite-dimensional representations of $Q$, then we have
\[
\langle\dim M,\dim N\rangle_Q=\dim_\CC \Hom_Q(M,N)-\dim_\CC \Ext_Q^1(M,N)
\]
In particular, the Tits form
\[
q_Q(\alpha):=\langle\alpha,\alpha\rangle_Q=\sum_{v\in Q_0}\alpha(v)^2-\sum_{a:v\rightarrow  w\in Q_1}\alpha(v)\alpha(w)
\]
satisfies
\[
 q_Q(\dim M)=\dim_\CC \End_Q(M)-\dim_\CC \Ext_Q^1(M,M).
\]
\end{lemma}

%%%%%%%%%%%%%%%%%%%%%%%%%%%%%%%%

\section{Veronese surface}

\begin{definition}
For a positive integer $d$, a coherent sheaf $\Ee$ on $\mathbb{P}^n$ of dimension $n$ is called {\it $d$-arithmetically Cohen–Macaulay} (abbreviated as ${d}$-aCM) if $\mathrm{H}^{i}(\mathbb{P}^n,\mathcal{E}\otimes\mathcal{O}_{\mathbb{P}^n}(dt))=0$ for all $0 < i < n$ and $t \in \mathbb{Z}$.
\end{definition}

From now on, we focus on the case $n = 2$. The fan of $\mathbb{P}^2$ has three rays, denoted by $\rho_0$, $\rho_1$, and $\rho_2$, whose primitive generators are $u_{\rho_0} = (1,0)$, $u_{\rho_1} = (0,1)$, and $u_{\rho_2} = (-1,-1)$, respectively. We denote the divisors corresponding to $\rho_0$, $\rho_1$, $\rho_2$ by $D_0$, $D_1$, $D_2$ respectively. For each cyclic permutation $(i,j,k)$ of $(0,1,2)$, we let $p_i := D_j \cap D_k$. Since each $D_i$ is a $T$-invariant line, the points $p_i$ are fixed under the torus action on $\mathbb{P}^2$. 

\begin{remark}\label{compability}
Note that any two finite filtrations of a finite-dimensional vector space are simultaneously splittable. Equivalently, by the relative-position description of pairs of flags (or the Bruhat decomposition), there is a basis adapted to both filtrations. This asserts the Klyachko compatibility on a two-dimensional cone so that one can always obtain a toric vector bundle on $\PP^2$ from the Klyachko decreasing filtrations.
\end{remark}

Then, a vector bundle $\Ee$ on $\PP^2$ is $d$-aCM if and only if we have
\[
\mathrm{H}^1(\PP^2, \Ee\otimes \Oo_{\PP^2}(dt))=0
\]
for all $t\in \ZZ$. From \cite{Kly90, HH} we can get the formula for the first cohomology of the toric bundle $\Ee$.
\\
\setcounter{equation}{1}
\begin{equation}
\mathrm{H}^{1}(\PP^2, \Ee)_m = 
\frac{E^{\mathbf{\rho}_{0}}_{m} \cap (E^{\mathbf{\rho}_{1}}_{m} +E^{\mathbf{\rho}_{2}}_{m})}{(E^{\mathbf{\rho}_{0}}_{m} \cap E^{\mathbf{\rho}_{1}}_{m})+(E^{\mathbf{\rho}_{0}}_{m} \cap E^{\mathbf{\rho}_{2}}_{m})} \label{eq:cohop2}
\end{equation}

\begin{definition}
For three subspaces $A_0,A_1,A_2\subseteq E$ with $a_i=\dim A_i$ for each $i$, we define
\begin{equation}\label{defect-quotient}
\Delta(A_0,A_1,A_2):= \dim \frac{A_0\cap(A_1+A_2)}{(A_0\cap A_1)+(A_0\cap A_2)}
\end{equation}
called the {\it distributivity defect}. If these subspaces are in general positions, we have
\begin{equation}\label{eta}
\begin{aligned}
\Delta(A_0,A_1,A_2)=\eta_r(a_0, a_1, a_2):={}&a_0+a_1+a_2 \min(r,a_0+a_1+a_2)\\
&-\max(0,a_0+a_1-r)-\max(0,a_1+a_2-r)\\
&-\max(0,a_2+a_0-r)+\max(0,a_0+a_1+a_2-2r).
\end{aligned}
\end{equation}
\end{definition}

\begin{remark}\label{Delta-property}
The expression measures the failure of the distributive identity
\[
A_0\cap(A_1+A_2)=(A_0\cap A_1)+(A_0\cap A_2).
\]
Thus $\Delta=0$ exactly when this identity holds for the ordered triple $(A_0;A_1,A_2)$. If one of $A_0,A_1,A_2$ is $0$ or $V$, then we have $\Delta(A_0,A_1,A_2)=0$. Consequently, only triples of proper nonzero filtration stages can contribute. By the dimension formula for a sum of two subspaces, we have
\begin{equation}\label{defect-dimension}
\begin{aligned}
\Delta(A_0,A_1,A_2)
={}&
\dim\bigl(A_0\cap(A_1+A_2)\bigr)+\dim(A_0\cap A_1\cap A_2)\\
&-\dim(A_0\cap A_1)-\dim(A_0\cap A_2).
\end{aligned}
\end{equation}
Formula \eqref{eta} is homogeneous under simultaneous scaling:
\[
\eta_{kr}(ka,kb,kc)=k\,\eta_r(a,b,c).
\]
This is the property used in scalable Klyachko flag constructions of
large-dimensional families of toric $d$-aCM bundles.
\end{remark}

\begin{theorem}\label{generalized-lemma}
Let $\Ee$ be a toric vector bundle of arbitrary rank $r$ on $\PP^2$. Then we have $\mathrm{H}^1(\Ee)\ne 0$ if and only if there exists a triple $(j_0,j_1,j_2)\in\mathbb \ZZ^{\oplus 3}$ with $j_0+j_1+j_2=0$ such that
\[
\Delta\left( E^{\rho_0}(j_0), E^{\rho_1}(j_1), E^{\rho_2}(j_2) \right)>0.
\]
Indeed, we have
\begin{equation}\label{weight-defect}
\dim \mathrm{H}^1(\Ee)_m=\Delta\left(E^{\rho_0}(j_0),E^{\rho_1}(j_1),E^{\rho_2}(j_2)\right)
\end{equation}
for $m=(j_0,j_1)\in M$ and $j_2=-j_0-j_1$,
\end{theorem}

\begin{proof}
For a character $m=(a,b)\in M\cong \mathbb \ZZ^{\oplus 2}$
\[
\mathrm{H}^1(\Ee)_m\simeq\frac{E^{\rho_0}(j_0)\cap \bigl(E^{\rho_1}(j_1)+E^{\rho_2}(j_2)\bigr)}{\bigl(E^{\rho_0}(j_0)\cap E^{\rho_1}(j_1)\bigr)
+\bigl(E^{\rho_0}(j_0)\cap E^{\rho_2}(j_2)\bigr)}
\]
for $(j_, j_1, j_2)=(a,b,-a-b)$. Then the assertion follows directly from Definition \ref{defect-quotient}. 
\end{proof}

\begin{corollary}\label{dacm}
A toric vector bundle $\Ee$ of rank $r$ on $\PP^2$ is $d$-aCM if and only if
\begin{equation}\label{dacm-defect}
\Delta\left(E^{\rho_0}(j_0),E^{\rho_1}(j_1),E^{\rho_2}(j_2)\right)=0
\end{equation}
for every triple $(j_0,j_1,j_2)\in \ZZ^{\oplus 3}$ with $j_0+j_1+j_2 \equiv 0 \pmod d$. 
\end{corollary}

\begin{proof}
Tensoring by a toric line bundle translates the indices in the three Klyachko filtrations by the coefficients of the corresponding invariant divisor. A twist by $\Oo_{\PP^2}(dt)$ changes the total index sum by $dt$, up to the harmless global sign determined by the chosen decreasing-filtration convention.

By Theorem \ref{generalized-lemma}, the group $\mathrm{H}^1(\Ee(dt))$ is nonzero precisely when there is a triple of filtration indices whose total sum is $dt$ and whose defect is positive. As $t$ ranges over $\ZZ$, these are exactly the triples whose index sum is congruent to zero modulo $d$.
\end{proof}

\begin{remark}
For ranks greater than two, the dimensions of the three filtration spaces do not in general determine the defect: their relative positions also matter.  For subspaces in general position, however, the defect has a closed dimension-only formula. For $r=2$, nonsplitting forces the relevant three lines to be mutually distinct, so the dimensions determine the defect. For $r>2$, two triples of subspaces can have the same dimensions but different intersection and sum dimensions. Therefore the exact arbitrary-rank criterion is Corollary~\ref{dacm}, formulated using $\Delta$. The function $\eta_r$ is valid only on the general-position locus.
\end{remark}

\begin{proposition}
The Veronese surface $\bigl(\PP^2,\Oo_{\PP^2}(d)\bigr)$ for $d=1,2$ is toric finite. 
\end{proposition}

\begin{proof}
By Horrocks' splitting criterion, a vector bundle $\Ee$ on $\PP^2$ is aCM with respect to $\Oo_{\PP^2}(1)$ if and only if it is a direct sum of line bundles. Thus $\Oo_{\PP^2}$ is the only indecomposable aCM bundle on $\PP^2$ for $d=1$, which has a toric structure. In the cae when $d=2$, there are only three indecomposable aCM bundles on $\PP^2$:
\[
\left\{\Oo_{\PP^2}, ~~ \Oo_{\PP^2}(1), ~~ \Omega_{\PP^2}^1(1) \right\}
\]
by \cite{MR}, all of which have toric structures. 
\end{proof}

%%%%%%%%%%%%%%%%%%%%%%%%%%%%%%%%%%%%%%%%%%%%%%%%%%%%%%%%%

\section{Geometric toric wildness}
\begin{lemma}\label{d3-filtration-acm}
Let $E_n$ be a vector space of dimension $6n$. Choose three flags
\[
\begin{cases}
W_i^n\subset W_i^{3n}\subset W_i^{5n}\subset E_n, \qquad {\text for }~~i=0,1,\\
W_2^{2n}\subset W_2^{4n}\subset E_n,
\end{cases}
\]
in simultaneous general position with $\dim W_i^k=k$. Equip them with the decreasing filtration indices displayed in Table~\ref{d3-klyachko-filtrations}. Then these filtrations define a toric vector bundle $\Ee_n$ of rank $6n$ on $\PP^2$, and
\[
\mathrm{H}^1\bigl(\Ee_n(3t)\bigr)=0 \qquad\text{for every }t\in\ZZ.
\]
In particular, $\Ee_n$ is $3$-aCM.
\end{lemma}

\begingroup
\renewcommand{\arraystretch}{1.5}
\begin{table}[ht]
\centering
\begin{tabular}{
    >{\centering\arraybackslash}m{1.2cm} |
    >{\centering\arraybackslash}m{1.35cm} |
    >{\centering\arraybackslash}m{1.35cm} |
    >{\centering\arraybackslash}m{1.35cm} |
    >{\centering\arraybackslash}m{1.35cm} |
    >{\centering\arraybackslash}m{1.35cm} |
    >{\centering\arraybackslash}m{1.35cm} |
    >{\centering\arraybackslash}m{1.35cm} |
    >{\centering\arraybackslash}m{1.35cm} 
}
\hline
 & $\cdots$ & $-4$ & $-3$ & $-2$ & $-1$ & $0$ & $1$ & $\cdots$ \\
\hline
$\rho_0$ &
$\cdots$ &
\cellcolor{lightred}$E_n$ &
\cellcolor{lightred}$E_n$ &
\cellcolor{lightblue}$W_0^{5n}$ &
\cellcolor{lightblue}$W_0^{3n}$ &
\cellcolor{lightblue}$W_0^{n}$ &
$0$ & $0$\\
\hline
$\rho_1$ &
$\cdots$ &
\cellcolor{lightred}$E_n$ &
\cellcolor{lightred}$E_n$ &
\cellcolor{lightblue}$W_1^{5n}$ &
\cellcolor{lightblue}$W_1^{3n}$ &
\cellcolor{lightblue}$W_1^{n}$ &
$0$ & $0$\\
\hline
$\rho_2$ &
$\cdots$ &
\cellcolor{lightred}$E_n$ &
\cellcolor{lightblue}$W_2^{4n}$ &
\cellcolor{lightblue}$W_2^{2n}$ &
$0$ &
$0$ &
$0$ &$0$ \\
\hline
\end{tabular}
\caption{The Klyachko filtrations for the $3$-aCM bundle of rank $6n$.}
\label{d3-klyachko-filtrations}
\end{table}
\endgroup

\begin{proof}
For a character $m\in M\cong \ZZ^{\oplus 2}$, let $A_i:=E^{\rho_i}_m$ be the filtration subspace selected at the ray $\rho_i$. The dimension of $\mathrm{H}^1(\Ee_n)_m$ is the distributivity defect $\Delta(A_0,A_1,A_2)$. By Remark \ref{Delta-property}, it is enough to consider triples of proper nonzero filtration stages. For subspaces in general position whose dimensions are $(na_0,na_1,na_2)$, homogeneity of the general-position defect gives
\[
\Delta(A_0,A_1,A_2)=\eta_{6n}(na_0,na_1,na_2)=n\eta_6(a_0,a_1,a_2).
\]
Thus it is enough to perform the calculation for $n=1$. The dimension-to-index correspondence for the prescribed filtrations in the case $n=1$ is as follows. 
\[
\begin{array}{c|c|ccc|}
\multirow{2}{*}{(i)  $~~\rho_0$, $\rho_1$}
 &\text{dimension}&1&3&5\\ \cline{2-5}
 &\text{index}&0&-1&-2
\end{array}
\qquad
\begin{array}{c|c|cc|}
\multirow{2}{*}{(ii)  $~~\rho_2$}
 &\text{dimension}&2&4\\ \cline{2-4}
 &\text{index}&-2&-3
\end{array}
\]
There are in total $18$ dimension triples, and a direct substitution
in the formula for $\eta_6$ gives the following table.
\[
\begin{array}{c|c|c||c|c|c}
(a_0,a_1,a_2)&\text{index sum}&\eta_6&
(a_0,a_1,a_2)&\text{index sum}&\eta_6\\ \hline
(1,1,2)&-2&0 &(3,1,2)&-3&0\\
(1,1,4)&-3&0 &(3,1,4)&-4&1\\
(1,3,2)&-3&0 &(3,3,2)&-4&2\\
(1,3,4)&-4&1 &(3,3,4)&-5&2\\
(1,5,2)&-4&1 &(3,5,2)&-5&1\\
(1,5,4)&-5&1 &(3,5,4)&-6&0\\
(5,1,2)&-4&1 &(5,1,4)&-5&1\\
(5,3,2)&-5&1 &(5,3,4)&-6&0\\
(5,5,2)&-6&0 &(5,5,4)&-7&0
\end{array}
\]
In particular, we have $\eta_6(a_0,a_1,a_2)>0$ only if the index sum is $4$ or $5$, neither of which is divisible by $3$. Since a weight $m$ can contrivute to $\mathrm{H}^1(\Ee_n(3t))$ only when the sum of its three filtration indices is congruent to $0$ modulo $3$, the bundle $\Ee_n$ is $3$-aCM. 
\end{proof}

\begin{remark}\label{quiver-presentation}
Use the \emph{increasing} reindexing
\[
F_i(j):=E^{\rho_i}(-j)
\]
of the Klyachko decreasing filtration to obtain the star quiver $Q_3$ with three arms of lengths $3,3,2$ and dimension vector $n\alpha_3$ for 
\[
\alpha_3= \bigl(6;\,1,3,5;\,1,3,5;\,2,4\bigr).
\]
Thus a representation $\Rr$ of dimension $n\alpha_3$ consists of vector spaces of these dimensions and linear maps along the arrows.
\end{remark}

Now let us deal with the case $d\ge 4$. Let $Q_d$ be the star quiver with three arms of length $3$ and dimension vector
\[
\alpha_{d}=\bigl(4;\,1,2,3;\,1,2,3;\,1,2,3\bigr).
\]
On every arm of $Q_d$ we place the stages of dimensions $n,2n,3n$ at the
increasing indices
\[
\varepsilon_d,\qquad \varepsilon_d+1,\qquad \varepsilon_d+2,
\]
where $\varepsilon_d$ is defined to be $1$ for $d=4$, or $0$ for $d \geq 5$. In the usual decreasing Klyachko convention, the same data are shown in Table~\ref{dge4-klyachko-filtrations} for a fixed vector space $E_n$ of dimension $4n$.

\begin{lemma}\label{dge4-filtration-acm}
Let $E_n$ be a vector space of dimension $4n$. Choose three flags
\[
W_i^n\subset W_i^{2n}\subset W_i^{3n}\subset E_n,
\]
for each $i=0,1,2$ in simultaneous general position with $\dim W_i^k=k$. Equip them with the decreasing filtration indices displayed in Table~\ref{dge4-klyachko-filtrations}. Then these filtrations define a toric vector bundle $\Ee_{n,d}$ of rank $4n$ on $\PP^2$, and
\[
\mathrm{H}^1\bigl(\Ee_{n,d}(dt)\bigr)=0 \qquad\text{for every }t\in\ZZ.
\]
In particular, $\Ee_{n,d}$ is $d$-aCM.
\end{lemma}

\begingroup
\renewcommand{\arraystretch}{1.5}
\begin{table}[ht]
\centering
\begin{tabular}{
    >{\centering\arraybackslash}m{1.25cm} |
    >{\centering\arraybackslash}m{1.35cm} |
    >{\centering\arraybackslash}m{1.55cm} |
    >{\centering\arraybackslash}m{1.55cm} |
    >{\centering\arraybackslash}m{1.55cm} |
    >{\centering\arraybackslash}m{1.55cm} |
    >{\centering\arraybackslash}m{1.55cm} |
    >{\centering\arraybackslash}m{1.55cm} 
}
\hline
 & $\cdots$ & $-\varepsilon_d-3$ & $-\varepsilon_d-2$
 & $-\varepsilon_d-1$ & $-\varepsilon_d$
 & $1-\varepsilon_d $ & $\cdots$ \\
\hline
$\rho_0$ &
$\cdots$ &
\cellcolor{lightred}$E_n$ &
\cellcolor{lightblue}$W_0^{3n}$ &
\cellcolor{lightblue}$W_0^{2n}$ &
\cellcolor{lightblue}$W_0^{n}$ &
$0$ & $0$\\
\hline
$\rho_1$ &
$\cdots$ &
\cellcolor{lightred}$E_n$ &
\cellcolor{lightblue}$W_1^{3n}$ &
\cellcolor{lightblue}$W_1^{2n}$ &
\cellcolor{lightblue}$W_1^{n}$ &
$0$ & $0$\\
\hline
$\rho_2$ &
$\cdots$ &
\cellcolor{lightred}$E_n$ &
\cellcolor{lightblue}$W_2^{3n}$ &
\cellcolor{lightblue}$W_2^{2n}$ &
\cellcolor{lightblue}$W_2^{n}$ &
$0$& $0$ \\
\hline
\end{tabular}
\caption{The decreasing Klyachko filtrations for the $d$-aCM bundle of rank $4n$.}
\label{dge4-klyachko-filtrations}
\end{table}
\endgroup

\begin{proof}
The proof goes in exactly same way as in the proof of Lemma \ref{d3-filtration-acm}. For subspaces in general position whose dimensions are $(na_0,na_1,na_2)$ with $a_i \in \{1,2,3\}$, homogeneity of the general-position defect gives
\[
\Delta(A_0,A_1,A_2)=\eta_{4n}(na_0,na_1,na_2)=n\eta_4(a_0,a_1,a_2).
\]
It is therefore enough to calculate in rank four, i.e. $n=1$. Since the value of $\eta_4(a_0,a_1,a_2)$ is symmetric in $a_0,a_1,a_2$, the complete calculation, written up to permutation, is as follows. 
\[
\begin{array}{c|c|c|c}
a_0+a_1+a_2&
\text{unordered triples }(a_0,a_1,a_2)&
\eta_4(a_0,a_1,a_2)&
\text{index sum}\\ \hline
5&(3,1,1),\ (2,2,1)&1&-3\varepsilon_d-2\\
6&(3,2,1)&1&-3\varepsilon_d-3\\
6&(2,2,2)&2&-3\varepsilon_d-3\\
7&(3,3,1),\ (3,2,2)&1&-3\varepsilon_d-4\\
\hline 
\star&\text{all other cases} &0&\text{unnecessary}
\end{array}
\]
In particular, we have $\eta_4(a,b,c)>0$ if and only if $a_0+a_1+a_2\in \{5,6,7\}$. If $d=4$, then $\varepsilon_d=1$, and the only positive-defect index sums are $\{-5,-6,-7\}$, none of which are divisible by $4$. If $d\geq5$, then $\varepsilon_d=0$, and the only positive-defect index sums are $\{-2,-3,-4\}$, again none of which are divisible by $d$. This implies that in every case when $\eta_{4n}(na_0,na_1,na_2)>0$ the corresponding index sum is nonzero modulo $d$. Thus we obtain $\mathrm{H}^1\bigl(\Ee_{n,d}(dt)\bigr)=0$ for every $t\in \ZZ$.
\end{proof}

Now for $d\ge 3$, write $Q_d$ for the relevant quiver and $\alpha_d$ for its dimension vector in Lemmas \ref{d3-filtration-acm} and \ref{dge4-filtration-acm}. 

\begin{lemma}\label{Schur-root}
For each $d\ge 3$ and any positive integer $n$, $n\alpha_d$ is a Schur root for $Q_d$. Furthermore, there exists a triple of flags of the dimension $n\alpha_d$ whose only common flag-preserving endomorphisms are scalar maps.
\end{lemma}

\begin{proof}
First assume that $d=3$. Let $E_1=\CC^6$ with basis $e_1,\ldots,e_6$, and take the first two flags to be
\[
\begin{aligned}
W_0^1&=\langle e_1\rangle,&
W_0^3&=\langle e_1,e_2,e_3\rangle,&
W_0^5&=\langle e_1,\ldots,e_5\rangle,\\
W_1^1&=\langle e_6\rangle,&
W_1^3&=\langle e_4,e_5,e_6\rangle,&
W_1^5&=\langle e_2,\ldots,e_6\rangle.
\end{aligned}
\]
An endomorphism preserving both opposite flags is block diagonal $D$ for the decomposition
\[
E_1=\langle e_1\rangle\oplus\langle e_2,e_3\rangle\oplus\langle e_4,e_5\rangle\oplus\langle e_6\rangle.
\]
Write it as $D=\mathrm{diag}(a,P,Q,f)$. For the third flag, set
\[
\begin{aligned}
W_2^2&=\langle u,v\rangle,&\text{ with } u=e_1+e_2+e_4, ~~v=e_3+e_5+e_6,\\
W_2^4&=\langle u,v,w_1,w_2\rangle, & \text{ with }w_1=e_1+e_3,~~ w_2=e_2+e_5.
\end{aligned}
\]
The condition $D(W_2^2)\subseteq W_2^2$ first gives
\[
 P=\begin{pmatrix}a&0\\0&f\end{pmatrix},
 \qquad
 Q=\begin{pmatrix}a&0\\0&f\end{pmatrix}.
\]
Directly from the four displayed generators, we have 
\[
W_2^4\cap E_a=\langle u\rangle,\qquad W_2^4\cap E_f=\langle v\rangle.
\]
If $a\ne f$ and $W_2^4$ were $D$-stable, the two spectral projections of $D$ would preserve $W_2^4$, so that
\[
W_2^4=(W_2^4\cap E_a)\oplus(W_2^4\cap E_f),
\]
which is impossible due to dimension counting. Hence $a=f$ and $D=aI_6$. 

Now assume that $n\ge 2$. From the Tits form value $q_{Q_3}(\alpha_3)=-2<0$, one can get that $\alpha_3$ is an imaginary non-isotropic Schur root. Applying Theorem~\ref{schofield-homogeneity} to $\alpha_3$, we obtain
\[
\mathrm{can}(n\alpha_3)=n\alpha_3,
\]
where the right-hand side is the single vector $n\alpha_3$, not $n$ separate copies of $\alpha_3$. Thus we get that $n\alpha_3$ is a Schur root for every $n\ge 1$. Equivalently, there exists a representation
\[
R_{3,n}\in\Rep(Q_3,n\alpha_3) \qquad\text{such that}\qquad \End_{Q_3} (R_{3,n})=\CC.
\]
In particular, the Schur locus 
\[
\Ss_n=\left\{R\in\Rep(Q_3,n\alpha_3)\mid\End_{Q_3}(R)=\CC\right\}
\]
is a nonempty open Zariski subset of irreducible affine space $\Rep(Q_3,n\alpha_3)$, because it contains $R_{3,n}$ and the dimension of the solution space of the linear equations defining $\End_{Q_3}(R)$ is upper semicontinuous in $R$. If we set 
\[
\Ii_n=\left\{R\in\Rep(Q_3,n\alpha_3)
   \mid\text{every arrow map of }R\text{ is injective}\right\},
\]
then it is also a nonempty open Zariski subset. Thus one can select 
\[
R_{3,n} \in \Ss_n\cap\Ii_n\neq\emptyset.
\]

Next, we assume that $d\ge 4$. One can set $E_1=\CC^4$ with basis $e_1,e_2,e_3,e_4$, and take the first flag to be the standard complete flag and the second to be the opposite complete flag:
\[
W_0^j=\langle e_1,\ldots,e_j\rangle, \qquad  W_1^j=\langle e_4,e_3,\ldots,e_{5-j}\rangle,
\]
for $j=1,2,3$. Then an endomorphism preserving both flags is diagonal, i.e., $D=\mathrm{diag}(\lambda_1,\lambda_2,\lambda_3,\lambda_4)$. For the third flag, begin with $W_2^1=\langle e_1+e_2+e_3+e_4\rangle$ and extend it arbitrarily to a complete flag $W_2^1\subset W_2^2\subset W_3^3$. Preservation of $W_2^1$ forces $\lambda_1=\lambda_2=\lambda_3=\lambda_4$, and so the three flags have only scalar common endomorphisms, proving that $\alpha_{d}$ is a Schur root. Since we have $q_{Q_{d}}(\alpha_{d})=-2$, this Schur root is again imaginary non-isotropic. Theorem~\ref{schofield-homogeneity} therefore implies that $n\alpha_{\geq4}$ is a Schur root for every $n\geq1$, and one can apply the previous argument again to obtain a triple $R_{d,n}$ of flags. 
\end{proof}

\begin{corollary}\label{geom-toric-wild}
The Veronese surface $\bigl( \PP^2, \Oo_{\PP^2}(d) \bigr)$ is geometrically toric wild for $d\ge 3$. 
\end{corollary}

\begin{proof}
We first deal with the case $d=3$. For each $n\geq1$, consider the irreducible parameter space
\[
\Xx_{3,n}=\mathrm{Fl}(n,3n,5n;6n)^2\times\mathrm{Fl}(2n,4n;6n).
\]
A point of $\Xx_{3,n}$ is precisely a triple of flags of the type used in Lemma~\ref{d3-filtration-acm}.  General position is a nonempty Zariski-open condition, so the lemma gives a nonempty open subset
\[
\Uu_{3,n}^{\mathrm{aCM}}\subseteq\Xx_{3,n}
\]
whose points define toric $3$-aCM bundles of rank $6n$. For a point $x\in \Xx_{3,n}$, the equivariant endomorphism algebra of the associated toric bundle $\Ee_x$ is identified with
\[
\End_T(\Ee_x) =\left\{A\in\End_{\CC}(E_n)\mid A(W_i^a)\subseteq W_i^a
       \text{ for every $i$ and $a$ }\right\}.
\]
Define a map 
\[
\Phi_x \colon \End_{\CC}(E_n)\longrightarrow \bigoplus_{i,a}\Hom_{\CC}(W_i^a,E_n/W_i^a), 
\]
by $A \mapsto \bigl(\pi_{i,a}\circ A_{|W_i^a}\bigr)_{i,a}$, then its kernel is exactly $\End_T(\Ee_x)$. In local coordinates on the flag varieties, $\Phi_x$ is represented by a matrix whose entries are regular functions of $x$ so that the function $x\mapsto\dim\ker\Phi_x$ is upper semicontinuous. Since $\CC I_{E_n}\subseteq\ker\Phi_x$ for every $x$, its smallest possible dimension is one. Thus the subset
\[
\Uu_{3,n}^{\mathrm{Sch}}:=\left\{x\in\Xx_{3,n}\mid\ker\Phi_x=\CC I_{E_n}\right\}
\]
is nonempty and Zariski open by Lemma \ref{Schur-root}. Now we set $\Uu_{3,n}=\Uu_{3,n}^{\mathrm{aCM}}\cap\Uu_{3,n}^{\mathrm{Sch}}$ the intersection of two nonempty open Zariski subsets. For $x\in\mathcal U_{3,n}$ we have $\End_T(\Ee_x)=\CC$. If $\Ee_x$ had a nontrivial decomposition in the equivariant category, projection onto one summand would be an idempotent in $\End_T(\Ee_x)$ different from $0$ and $\mathrm{id}$. Therefore $\Ee_x$ is equivariantly indecomposable.

Note that a partial flag with block sizes $b_1,\ldots,b_s$ has dimension
\[
\sum_{p<q}b_pb_q,
\]
from which we obtain $\dim\Xx_{3,n}=13n^2+13n^2+12n^2=38n^2$. The group $\operatorname{PGL}_{6n}$ acts by simultaneous change of basis. On $\mathcal U_n$ the stabilizer in $\operatorname{PGL}_{6n}$ is trivial, because the common endomorphism algebra before dividing by scalars is $\CC$. Thus we have
\[
\dim(\Uu_{3,n}/\mathrm{PGL}_{6n})=38n^2-\bigl((6n)^2-1\bigr)=2n^2+1,
\]
proving the geometric toric wildness for $d=3$. 

For $d\ge 4$, one sets the irreducible parameter space
\[
\Xx_{d,n}=\mathrm{Fl}(n,2n,3n;4n)^3.
\]
Then one can similarly define a nonempty Zariski open subset $\Uu_{d,n}\subset \Xx_{d,n}$, whose points correspond to indecomposable $d$-aCM bundles of rank $4n$. Finally the assertion follows from
\[
\dim(\Uu_{d,n}/\mathrm{PGL}_{4n})=18n^2-\bigl((4n)^2-1\bigr)=2n^2+1. \qedhere
\] 
\end{proof}

\begin{remark}
In fact, it is shown in Theorem \ref{main2} that the Veronese surface $\bigl( \PP^2, \Oo_{\PP^2}(d)\bigr)$ for $d \ge 3$ is toric wild. Although the geometric toric wildness is automatically implied from the toric wildness, it is still needed to show the existence of a representation $R_{d,1}\in \Rep(Q_d, \alpha_d)$ to prove Theorem \ref{main2}, which was guaranteed from Lemma \ref{Schur-root}. 
\end{remark}

\begin{proposition}\label{d3-stable}
The Schur bundles $\Ee_d$ corresponding to $R_{d,1}$ for $d\ge 3$ in Lemmas \ref{d3-filtration-acm} and \ref{dge4-filtration-acm} with flags in general position, are slope-stable with respect to $\Oo_{\PP^2}(1)$.
\end{proposition}

\begin{proof}
First assume that $d=3$. From Table~\ref{d3-klyachko-filtrations} one obtains $c_1(\Ee_3)=-36$ and $\mu (\Ee_3)=-6$. For the computational convenience, let us shift the the first two rays by $3$ and the third ray by $4$, by tensoring $\Oo_{\PP^2}(10)$. Let $0\neq S\subsetneq E_1=:E$ be a vector subspace with $\dim S=:s$, and define
\[
\begin{split}
w(S):={}&\sum_{i=0}^1\bigl( \dim(S\cap W_i^1)+\dim(S\cap W_i^3) +\dim(S\cap W_i^5)\bigr)\\
&+\dim(S\cap W_2^2)+\dim(S\cap W_2^4).
\end{split}
\]
Then the Klyachko slope criterion in \cite{DDK20, DDK21} asks us to prove
\begin{equation}\label{Klyachoko-stability}
\frac{w(S)}s<\frac{w(E)}6=4 
\end{equation}
for every proper nonzero $S\subset E$. We now establish \eqref{Klyachoko-stability} by a finite Schubert calculation. Extend the three partial flags to complete flags. Then the Schubert position of an $s$-plane relative to one of them is indexed by $I=\{i_1<\cdots<i_s\}\subset\{1,\ldots,6\}$, and set
\begin{align*}
c_s(I)&:=\sum_{a=1}^s(6-s+a-i_a),\\
A(I)&:=\sum_{p\in\{1,3,5\}}\#\{a:i_a\leq p\},\\
B(I)&:=\sum_{p\in\{2,4\}}\#\{a:i_a\leq p\}.
\end{align*}
Here, $c_s(I)$ is the codimension of the corresponding Schubert variety, while $A(I)$ and $B(I)$ are precisely the contributions to $w(S)$ from a flag of the first and third type, respectively. For three general flags, the three Schubert conditions indexed by $I_0,I_1,I_2$ can have a common solution only if
\begin{equation}\label{common-sol}
c_s(I_0)+c_s(I_1)+c_s(I_2)\leq\dim\mathrm{Gr}(s,6)=s(6-s).
\end{equation}
Under this constraint, direct enumeration of the subsets of
$\{1,\ldots,6\}$ gives the following complete list of maximal
values. 
\[
\begin{array}{c|c|c|c}
s&\max\bigl(A(I_0)+A(I_1)+B(I_2)\bigr)
 &4s&\text{one maximizing }(I_0,I_1,I_2)\\ \hline
1&3&4&(\{1\},\{6\},\{6\})\\
2&7&8&(\{1,3\},\{5,6\},\{4,6\})\\
3&11&12&(\{1,3,5\},\{1,5,6\},\{4,5,6\})\\
4&15&16&(\{1,2,3,5\},\{3,4,5,6\},\{2,4,5,6\})\\
5&19&20&(\{1,2,3,4,5\},\{2,3,4,5,6\},
                 \{2,3,4,5,6\}).
\end{array}
\]
Thus $w(S)\leq4s-1<4s$ for $1\leq s\leq5$, proving \eqref{Klyachoko-stability}. 

Now assume that $d\ge 4$, and twist by $\Oo_{\PP^2}(3\varepsilon_d+3)$ to have all four filtration weights $3,2,1,0$. As in the case $d=3$, we set
\[
w(S):= \sum_{i=0}^2\sum_{p=1}^3\dim(S\cap W_i^p).
\]
for a nonzero proper subspace $S\subset E_1=:E$ of $\dim S=:s$ to have the Klyachko stability criterion $w(S)/s<w(E)/4=9/2$. We verify the inequality uniformly for general flags. Relative to one complete flag, a Schubert position in $\mathrm{Gr}(s,4)$ is indexed by $J=\{j_1<\cdots<j_s\}\subset\{1,2,3,4\}$. Set
\[
 c_s(J_i):=\sum_{a=1}^s(4-s+a-j_{i,a}).
\]
for the index sets $J_i=\{j_{i,1}<\cdots<j_{i,s}\}\subset\{1,2,3,4\}$ for each $i=0,1,2$. This is the codimension of the Schubert variety determined by $J_i$. Moreover, we have
\[
\sum_{p=1}^3\dim(S\cap W_i^p)=\sum_{a=1}^s(4-j_{i,a})=c_s(J_i)+\frac{s(s-1)}2.
\]
For three general flags, the three Schubert conditions can have a
common solution only if
\[
c_s(I_0)+c_s(I_1)+c_s(I_2) \leq\dim\mathrm{Gr}(s,4)=s(4-s).
\]
Therefore we have
\[
w(S)\leq s(4-s)+\frac{3s(s-1)}2<\frac{9}{2}s
\]
for $s=1,2,3$.
\end{proof}

%%%%%%%%%%%%%%%%%%%%%%%%%%%%%%%%%%%%%%%%%%%%%

\section{Toric wildness}

\begin{definition}\label{filt}
For a fixed representation $R_d:=R_{d,1}$ in the proof of Lemma \ref{Schur-root}, the category $\mathbf{Filt}_{Q_d}(R_d)$ is defined to be the full subcategory of $\Rep(Q_d)$, consisting of representations $M$ admitting a finite filtration by subrepresentations
\[
0=M_0\subset M_1\subset\cdots\subset M_\ell=M
\]
such that $M_j/M_{j-1}\cong R_d$ for each $j$. Equivalently, it is the smallest full extension-closed subcategory containing $R_d$.
\end{definition}

\begin{lemma}\label{extension-closure}
Let $\Ee_d$ be the toric $d$-aCM vector bundle corresponding to the flag representation $R_d$. Every object $M\in\Filt_{Q_d}(R_d)$ is a flag representation and it corresponds to a toric $d$-aCM vector bundle $\Ee_M$. In fact, $\Ee_M$ is an iterated equivariant self-extension of $\Ee_d$.  
\end{lemma}

\begin{proof}
For a short exact sequence in $\Rep(Q_d)$
\[
0\to A \xrightarrow{\iota} B \xrightarrow{\pi} C \to 0
\]
and for each arrow $a:u\rightarrow v$ of $Q_d$, there is a commutative diagram
\[
\begin{tikzcd}[column sep=3.1em]
0 \arrow[r] &
A(u) \arrow[r,"\iota_u"] \arrow[d,"A(a)"'] &
B(u) \arrow[r,"\pi_u"] \arrow[d,"B(a)"'] &
C(u) \arrow[r] \arrow[d,"C(a)"'] &
0\\
0 \arrow[r] &
A(v) \arrow[r,"\iota_v"] &
B(v) \arrow[r,"\pi_v"] &
C(v) \arrow[r] &
0 .
\end{tikzcd}
\]
If $A(a)$ and $C(a)$ are injective, then $B(a)$ is also injective by the Snake lemma. In each step of $R_d$-filtration of $M$, one has a short
exact sequence
\begin{equation}\label{iteration}
0\to M_{j-1}\to M_j\to R_d\cong M_1\to 0.
\end{equation}
Since the representation $R_d$ has injective maps along every arm of $Q_d$, one can conclude that each $M_j$ is a flag representation by applying the previous argument inductively. Setting $E_j=M_j (v_{\infty})$ for the central vertex $v_{\infty}\in Q$, each $M_j$ corresponds to a toric bundle $\Ee_j$ whose associated Klyachko filtrations are $\{E_j^{\rho_i}(t)\}_{t\in\mathbb Z}$ with subspaces of $E_j$. At every vertex on each arm, we have the corresponding exact sequence
\[
0\to E_{j-1}^{\rho_i}(t)\to E_j^{\rho_i}(t)\to E_{1}^{\rho_i}(t)\to 0,
\]
yielding an exact sequence of $T$-equivariant vector bundles 
\begin{equation}\label{iteration-bundle}
0\to \Ee_{M_{j-1}}\to \Ee_{M_j} \to \Ee_d\cong \Ee_{M_1} \to 0.
\end{equation}
Since the $d$-aCM condition is closed under such extensions, each $\Ee_{M_j}$ is $d$-aCM by induction and so is $\Ee_M$.
\end{proof}

\begin{theorem}\label{embedding}
Let $Q=(Q_0, Q_1)$ be a finite acyclic quiver and let $R\in\Rep(Q)$ satisfy
\[
\End_Q(R)=\CC, \qquad \dim_\CC \Ext_Q^1(R,R)\geq3.
\]
Then there is a $\CC$-linear exact faithful functor
\[
\Phi_R: \mathrm{mod}\CC\langle x,y\rangle \longrightarrow \Filt_Q(R)
\]
which preserves indecomposable objects and reflects isomorphism classes.
\end{theorem}

\begin{proof}
Since $Q$ has no relations, we have $\Ext_Q^1(R,R)\cong \mathrm{coker}(\delta)$, where $\delta$ is the map 
\[
\delta : C^0(R,R)= \bigoplus_{v\in Q_0}\End_\CC(R(v)) \longrightarrow  C^1(R,R) = \bigoplus_{a:v\rightarrow  w\in Q_1} \Hom_\CC(R(v),R(w)),
\]
defined by $\delta\bigl((h_v)_v\bigr)=\bigl(h_wR(a)-R(a)h_v\bigr)_{a~:~v\rightarrow w}$. Choose three linearly independent classes and corresponding cocycle representatives:
\[
\xi_0,\xi_x,\xi_y\in\Ext_Q^1(R,R), \qquad z_0,z_x,z_y\in C^1(R,R).
\]
Let $M=(W;X,Y)$ be a finite-dimensional $\CC\langle x,y\rangle$-module, where $X,Y\in\End_\CC(W)$ are the actions of $x$ and $y$ on a vector space $W$. Then we define a functor $\Phi_R$ by $M \mapsto \Phi_R(M)$ such that
\begin{itemize}
    \item [(i)] $\Phi_R(M)(v)=(R(v)\otimes_\CC W)\oplus(R(v)\otimes_\CC W)$ for $v\in Q_0$; \vspace{.1cm}
    \item [(ii)] $ \Phi_R(M)(a)=\begin{pmatrix}
R(a)\otimes I_W&
z_0(a)\otimes I_W+z_x(a)\otimes X+z_y(a)\otimes Y\\
0&R(a)\otimes I_W
\end{pmatrix}$ for $a\in Q_1$.
\end{itemize}
Since $Q$ has no relations, these arrow matrices define a representation. If $f:M=(W;X,Y)\rightarrow M'=(W';X',Y')$ is a module homomorphism, then we obtain
\begin{equation}\label{inter-eq}
fX=X'f,\qquad fY=Y'f.
\end{equation}
If we define
\[
\Phi_R(f)(v)=\begin{pmatrix}I_{R(v)}\otimes f&0\\0&I_{R(v)}\otimes f \end{pmatrix},
\]
then \eqref{inter-eq} shows directly that these block matrices commute with every arrow map, and so one can obtain a $\CC$-linear functor
\[
\Phi_R : \mathrm{mod}\CC\langle x,y\rangle \longrightarrow \Rep(Q).
\]
It is faithful because $I_{R(v)}\otimes f=0$ for every $v$ forces $f=0$ whenever $R\ne0$. There also exists a natural exact sequence
\begin{equation}
0\to  R\otimes_\CC W \xrightarrow{i_M}\Phi_R(M) \xrightarrow{p_M}R\otimes_\CC W \to 0,
\end{equation}
which corresponds to the class $\eta_M=\xi_0\otimes I_W+\xi_x\otimes X+\xi_y\otimes Y$. Set $n=\dim_\CC W$ and choose a complete flag $0=W_0\subset W_1\subset\cdots\subset W_n=W$ to have the filtration for the representation $R \otimes_{\CC}W$:
\[
0=R\otimes W_0 \subset R\otimes W_1\subset\cdots\subset R\otimes W_n=R\otimes W
\]
with each successive quotients isomorphic to $R$. Set
\[
A=\mathrm{Im}(i_M)=\ker (p_M)\subset\Phi_R(M)
\]
be the distinguished subrepresentation and define
\[
A_i=i_M(R\otimes W_i), \qquad B_j=p_M^{-1}(R\otimes W_j)
\]
for $0\le i,j\le n$. Then we obtain
\[
0=A_0\subset A_1\subset\cdots\subset A_n=B_0 \subset B_1\subset\cdots\subset B_n=\Phi_R(M)
\]
with each successive quotient isomorphic to $R$. Thus by induction on $n$ we can obtain that
\[
\Phi_R(M)\in\Filt_Q(R).
\]
Since $\Filt_Q(R)$ is a full subcategory of $\Rep(Q)$, the quiver morphism $\Phi_R(f)$ constructed above is automatically a morphism in $\Filt_Q(R)$. Therefore one can obtain a functor
\[
\mathrm{mod}\CC \langle x,y\rangle \xrightarrow{\ \Phi_R\ } \Filt_Q(R) \hookrightarrow \Rep(Q).
\]
By the standard argument one can also show that $\Phi_R$ is exact. Indeed, a short exact sequence of finite-dimensional $\CC\langle x,y\rangle$-modules gives an exact sequence of the underlying vector spaces. At each vertex $v\in Q_0$, the functor $\Phi_R$ is the direct sum of two copies of the exact functor $R(v)\otimes_\CC -$. Its image is consequently exact at every vertex. Exactness in $\Rep(Q)$ is computed vertexwise, so the image sequence is exact in $\Rep(Q)$. Since $\Filt_Q(R)$ is extension-closed, it carries the exact structure inherited from $\Rep(Q)$. Hence the same sequence is exact in $\Filt_Q(R)$.

Hence it remains to show that $\Phi_R$ preserves indecomposable objects and reflects isomorphism classes, for which we observe the following assertion. 

\begin{claim}\label{inverse}
Any morphism $\psi:\Phi_R(M)\rightarrow \Phi_R(M')$ for $M=(W;X,Y)$ and $M'=(W';X',Y')$, induces a module homomorphism $f: M \rightarrow M'$. 
\end{claim}
\begin{claimproof}
The composite
\[
g=p_{M'}\circ\psi \circ i_M: R\otimes_\CC W\longrightarrow R\otimes_\CC W'
\]
is a morphism in $\Rep(Q)$. Note that there is a unique linear map $g_0:W\rightarrow W'$ such that $g=I_R\otimes g_0$, because $\End_Q(R)=\CC$. Since the morphism $\psi \circ i_M$ is a lift of $g$ through $p_{M'}$, i.e., $p_{M'}(\psi \circ i_M)=g$, the pullback of the target extension corresponding to $\eta_{M'}$ along $g$ split so that we have
\[
g^* \eta_{M'}=\xi_0\otimes g_0+\xi_x\otimes X'g_0+\xi_y\otimes Y'g_0=0.
\]
Because $\xi_0,\xi_x,\xi_y$ are linearly independent, we have $g_0=0$ by comparing their coefficients. Thus we have $p_{M'}\circ \psi \circ i_M=0$ and so one can obtain the restriction map
\[
\psi_{\mathrm{res}} : R\otimes_\CC W \longrightarrow R\otimes_{\CC} W'.
\]
This also yields a map on quotients
\[
\overline{\psi} : R\otimes_\CC W \cong \frac{\Phi_R(M)}{i_M(R\otimes_\CC W)} \longrightarrow  R\otimes_\CC W' \cong \frac{\Phi_R(M')}{i_{M'}(R\otimes_\CC W')}.
\]
Both maps have the forms $I_R\otimes f_0$ and $I_R\otimes f_1$, respectively for uniquely determined linear maps $f_0,f_1:W\rightarrow  W'$. Compatibility for a morphism of extensions gives
\[
(f_0)_*\eta_M=(f_1)^*\eta_{M'}.
\]
In terms of the three independent extension classes, this becomes
\[
\xi_0\otimes f_0+\xi_x\otimes f_0X+\xi_y\otimes f_0Y=\xi_0\otimes f_1+\xi_x\otimes X'f_1+\xi_y\otimes Y'f_1.
\]
Comparison of coefficients gives
\[
 f_0=f_1=:f,\qquad fX=X'f,\qquad fY=Y'f.
\]
In particular, $f:M\rightarrow M'$ is a $\CC\langle x,y\rangle$-module homomorphism. 
\end{claimproof}

\noindent The difference $\psi-\Phi_R(f)$ induces a map
\[
\frac{\Phi_R(M)}{i_M(R\otimes_\CC W)} \longrightarrow R\otimes_{\CC}W' \subset \Phi_R(M')
\]
so that we have $\psi-\Phi_R(f)=i_{M'}\circ (I_R\otimes h)\circ p_M=:k_h$ for some $h\in\Hom_\CC(W,W')$. Then the map $k_h$ corresponds to the strictly upper-triangular block matrix
\[
k_h(v)=\begin{pmatrix}0&I_{R(v)}\otimes h\\0&0 \end{pmatrix},
\]
and so $\psi$ has a unique decomposition as $\Phi_R(f)\oplus k_h$. 
Assume that another morphism $\varphi: \Phi_R(M')\rightarrow \Phi_R(M'')$ induces a module homomorphism $g: M' \rightarrow M''$ with $k_{\ell}=\varphi-\Phi_R(g)$. Then we have 
\begin{equation}\label{square-free}
k_\ell\circ k_h =\left( i_{M''}\circ(I_R\otimes\ell)\circ p_{M'}\right)\circ \left(i_{M'}\circ (I_R\otimes h)\circ p_M \right) =0,
\end{equation}
because of $p_{M'}\circ i_{M'}=0$. One can also observe that 
\begin{equation}\label{mixed}
\Phi_R(g)k_h=k_{gh}, \qquad k_\ell\Phi_R(f)=k_{\ell f}
\end{equation}
from the natural identity
\[
\Phi_R(g)i_{M'}=i_{M''}(I_R\otimes g), \qquad p_{M'}\Phi_R(f)=(I_R\otimes f)p_M.
\]

Finally, suppose that $\psi:\Phi_R(M)\xrightarrow{~~\cong~~}\Phi_R(M')$. Applying Claim \ref{inverse} and the argument right in the above both to $\psi$ and to $\varphi=\psi^{-1}$, we have
\[
 \begin{aligned}
 I_{\Phi_R(M)}=\psi^{-1}\psi
 &=\bigl(\Phi_R(g)+k_\ell\bigr)
   \bigl(\Phi_R(f)+k_h\bigr)\\
 &=\Phi_R(gf)+k_{gh}+k_{\ell f}+k_\ell k_h\\
 &=\Phi_R(gf)+k_{gh+\ell f},
 \end{aligned}
\]
and the uniqueness of the decomposition gives $gf=I_W$. Similarly, we have $fg=I_{W'}$ so that we have $f: M\xrightarrow{\cong} M'$. Since the converse is immediate, one can conclude that $\Phi_R$ reflects isomorphism classses. 

Assume that $M$ is indecomposable, and let $e\in\End_Q(\Phi_R(M))$ be an idempotent. By Claim \ref{inverse} and the previous argument, it induces an idempotent module endomorphism $f\in\End_{\CC \langle x,y\rangle}(M)$. A finite-dimensional indecomposable module has no nontrivial idempotent endomorphism, so we have either $f=0$ or $f=\id$. If $f=0$, then $e=k_h$ for some $h$, and this gives $e=e^2=k_h\circ k_h=0$. If $f=\id$, apply the same argument to the idempotent $\id -e$, whose induced module endomorphism is zero. This gives $\id-e=0$, and so $\Phi_R(M)$ has no nontrivial idempotent endomorphism and is indecomposable.
\end{proof}

\begin{theorem}\label{main2}
The Veronese surface $\bigl(\PP^2,\Oo_{\PP^2}(d)\bigr)$ for $d\ge 3$ is toric wild.
\end{theorem}

\begin{proof}
Choose $R=R_d$ from Proposition~\ref{Schur-root} and definition \ref{filt}. Since $q_{Q_d}(\alpha_d)=-2$, we have $\dim_{\CC} \Ext_{Q_d}^1(R_d, R_d)=3$ and so Theorem \ref{embedding} gives an exact representation embedding
\[
\Phi_{R_d}: \mathrm{mod}\CC \langle x,y\rangle \longrightarrow \Filt_{Q_d} (R_d).
\]
By Lemma~\ref{extension-closure}, the target consists of toric $d$-aCM vector bundles. Thus the composite
\[
\Phi_d : \mathrm{mod}\CC\langle x,y\rangle\longrightarrow\aCM_T(\PP^2, \Oo_{\PP^2}(d))
\]
is exact, preserves indecomposability, and reflects equivariant isomorphism classes. Recall from Lemma \ref{extension-closure} that $\Ee_d$ is the $d$-aCM toric bundle associated to $R_d$ with $r_d:=\mathrm{rank}(\Ee_d)$. Thus we have $\mathrm{rank}(\Phi_d(M))=2r_d\dim_\CC M$. Consequently, if $\Phi_d(M)\sim_{\Oo_{\PP^2}(d)}\Phi_d(N)$, we obtain
\[
\dim_\CC M=\dim_\CC N=:n.
\]
For fixed $n>0$, all such representations $\Phi_R(M)$ have the same dimension vector $2n\alpha_d$ and the same filtration jump indices. In particular, their underlying bundles have the same first Chern class and their equivariant determinants have the same normalized linearization. Thus we obtain $\Phi_d(N)\cong_T\Phi_d(M)\otimes\CC_\chi$. Taking equivariant determinants gives
\[
\det\Phi_d(N)\cong_T \det\Phi_d(M)\otimes\CC_{\,2nr_d\chi}.
\]
Since the two normalized equivariant determinants are equal, we have $\chi=0$. Thus $\Phi_d(N)\cong_T\Phi_d(M)$, and the isomorphism-reflection property of Theorem~\ref{embedding} yields $N\cong M$. The converse is immediate. Hence the functor reflects equivalence classes modulo $\sim_{\Oo_{\PP^2}(d)}$, exactly as required in the definition.
\end{proof}

\begin{remark}
The strategy in Theorems \ref{embedding} and \ref{main2} follow the one in \cite[Section 3]{FPL}, where they consider representations of $\omega$-Kronecker quiver for $\Ext_X^1(\Bb, \Aa)$ of dimension $\omega\ge 3$ for two simple aCM sheaves $\Aa$ and $\Bb$, to deduce the wildness of $X$. 
\end{remark}

Let $\Ulr_T\bigl(\PP^2, \Oo_{\PP^2}(d) \bigr)$ denote the exact category of $T$-equivariant Ulrich bundles on $\bigl(\PP^2,\Oo_{\PP^2}(d)\bigr)$. We call the polarized surface \emph{toric Ulrich-wild} if there is an exact
representation embedding
\[
\mathrm{mod}\CC\langle x,y\rangle \longrightarrow\Ulr_T(\PP^2,\Oo_{\PP^2}(d))
\]
which preserves indecomposability and reflects the relevant equivariant equivalence classes. It follows directly from the definition that $\Ulr_T\bigl(\PP^2, \Oo_{\PP^2}(d) \bigr)$ is closed under finite direct sums and equivariant extensions.

\begin{lemma}\label{Ulrich}
For the Schur bundles $\Ee_d$ corresponding to $R_{d,1}$ for $d\ge 3$ in Lemmas \ref{d3-filtration-acm} and \ref{dge4-filtration-acm} with flags in general position, set
\[
\Uu_3:=\Ee_3(9), \qquad \Uu_4:=\Ee_4(12).
\]
Then $\Uu_d$ is Ulrich with respect to $\Oo_{\PP^2}(d)$ for $d=3,4$.
\end{lemma}

\begin{proof}
On $\bigl( \PP^2, \Oo_{\PP^2}(d) \bigr)$, the bundle $\Uu_d$ is Ulrich if and only if
\[
\mathrm{H}^\bullet(\Uu_d(-d)) \cong 0, \qquad \mathrm{H}^\bullet(\Uu_d(-2d)) \cong 0,
\]
where the case $\bullet=1$ is automatic by Lemmas \ref{d3-filtration-acm} and \ref{dge4-filtration-acm}. Consider first $d=3$. From
\[
\mu(\Uu_3(-6))=-3, \qquad \mu(\Uu_3(-3))=0.
\]
Stability gives the vanishing for $\bullet=0$. On the other hand, the Serre duality gives
\[
\begin{aligned}
\mathrm{H}^2(\Uu_3(-6))&\cong \Hom(\Uu_3, \Oo_{\PP^2}(3))^\vee \cong0;\\
\mathrm{H}^2(\Uu_3(-3))&\cong  \Hom(\Uu_3,\Oo_{\PP^2})^\vee \cong 0,
\end{aligned}
\]
because of $\mu(\Uu_3)=3$. Now assume $d=4$ and then we have
\begin{equation}\label{euler}
\chi(\Uu_4(-4))=\chi(\Uu_4(-8))=0
\end{equation}
by the Riemann-Roch. Since $\Uu_4$ is stable, the Serre duality and $\mu(\Uu_4)=18$ gives
\[
\mathrm{H}^2(\Uu_4(-4)) \cong \Hom(\Uu_4, \Oo_{\PP^2}(1))^\vee \cong 0,
\]
and this gives $\mathrm{H}^0(\Uu_4(-4))=0$ by \eqref{euler}. Similarly, stability gives $\mathrm{H}^0(\Uu_4(-8)) \cong 0$ and so we obtain $\mathrm{H}^2(\Uu_4(-8))=0$ by \eqref{euler}. 
\end{proof}

\begin{proof}[Proof of Theorem \ref{ulrich-wild}]
Define
\[
\Psi_d : \mathrm{mod} \CC \langle x,y \rangle \longrightarrow \aCM_T(\PP^2, \Oo_{\PP^2}(d))
\]
by $\Psi_d(M)=\Ee_M\otimes \Oo_{\PP^2}(m_d)$, where $\Ee_M$ is as in Lemma \ref{extension-closure}, $m_3=9$ and $m_4=12$. Indeed, if $M=(W;X,Y)$ with $n:=\dim_\CC W$, then $\Psi_d(M)$ lies in an equivariant exact sequence
\[
0\to \Uu _d\otimes_\CC W \to \Phi_{\Uu_d}(M) \to \Uu_d\otimes_\CC W \to 0
\]
so that it is Ulrich with 
\[
\mathrm{rank}(\Phi_{\mathcal U_d}(M))=2r_dn =
\begin{cases} 12n,&d=3,\\ 8n,&d=4. \end{cases}
\]
Thus $\Psi_d$ factors through $\Ulr_T(\PP^2, \Oo_{\PP^2}(d))$. Then the conditions for the toric Ulrich-wildness is obtained from the conditions for the toric-wildness: Theorem~\ref{embedding} also shows that the functor preserves indecomposability and reflects equivariant isomorphism classes. 
\end{proof}

\begin{conjecture}\label{conj}
For every integer $d>4$, the Veronese surface $\bigl(\PP^2,\Oo_{\PP^2}(d)\bigr)$ is toric Ulrich-wild. 
\end{conjecture}

\begin{remark}
Theorem~\ref{ulrich-wild} proves the analogous statement for $d=3,4$. The bundles $\Uu_3$ and $\Uu_4$ are not Ulrich for $d>4$: its Euler polynomial is
\[
\chi(\Uu_d(t))=\frac{\mathrm{rank}(\Uu_d)}{2}(t+d)(t+2d)
\]
for $d=3,4$ so that two roots differ exactly by $d$. Thus Conjecture~\ref{conj} requires a new equivariant Ulrich bundle to start with, or a different representation embedding, rather than a
formal rescaling of the present construction.
\end{remark}
%%%%%%%%%%%%%%%%%%%%%%%%%%%%%%

\end{document}